\documentclass{amsart}

\usepackage[active]{srcltx}

\usepackage{graphicx}

\usepackage{latexsym}
\usepackage{amsmath,amsthm}
\usepackage{amsfonts}
\usepackage[psamsfonts]{amssymb}
\usepackage{enumerate}

\usepackage[final]
{showkeys}
\usepackage{url}

\RequirePackage[OT1]{fontenc}

\usepackage{calrsfs}

\usepackage{tikz}
\usetikzlibrary{decorations}
\usetikzlibrary{decorations.pathreplacing}

\usepackage{hyperref}

\theoremstyle{plain} 
\newtheorem{theorem}{Theorem}
\newtheorem{corollary}[theorem]{Corollary}

\newtheorem*{conjecture*}{Conjecture}
\theoremstyle{definition} 

\theoremstyle{definition}

\newtheorem*{remark*}{Remark}

\makeatletter
  \renewcommand\section{\@startsection {section}{1}{\z@}%
                                   {-\bigskipamount}%
                                   {\medskipamount}%
                                   {\large\bfseries
                                   \raggedright}}

  \renewcommand\subsection{\@startsection {subsection}{2}{\z@}%
                                   {-\medskipamount}%
                                   {\smallskipamount}%
                                   {\bfseries
                                   \raggedright}}
\makeatother

\renewcommand{\gg}{>\kern-2pt>}
\renewcommand{\ll}{<\kern-2pt<}

\renewcommand{\gg}{>\kern-2pt>}
\renewcommand{\ll}{<\kern-2pt<}

\newcommand{\dd}{\partial}
\renewcommand{\dd}{\operatorname{d}\!}

\newcommand{\ga}{\gamma}
\newcommand{\Ga}{\Gamma}

\newcommand{\D}{\mathcal{D}}

\newcommand{\LD}{\mathcal{L}\!\mathcal{D}}
\renewcommand{\LD}{\mathcal{L}{\kern -1.9pt}\mathcal{D}}
\renewcommand{\LD}{\mathcal{D}}
\renewcommand{\LD}{\mathcal{L}{\kern -4pt}\mathcal{C}}
\renewcommand{\LD}{\mathcal{R}{\kern -3pt}\mathcal{C}}

\renewcommand{\P}{\operatorname{\mathsf{P}}} 

\newcommand{\E}{\operatorname{\mathsf{E}}}

\newcommand{\R}{{\mathbb{R}}}
\newcommand{\C}{\mathbb{C}}

\newcommand{\tE}{{\tilde{E}}}

\renewcommand{\D}{\mathrel{\overset{\operatorname{D}}=}}

\newcommand{\vp}{\varepsilon}

\begin{document}


\title{The exp-normal distribution is infinitely divisible}


\author{Iosif Pinelis}

\address{Department of Mathematical Sciences\\
Michigan Technological University\\
Hough\-ton, Michigan 49931, USA\\
E-mail: ipinelis@mtu.edu}



%
%

\begin{abstract}
Let $Z$ be a standard normal random variable (r.v.). It is shown that the distribution of the r.v.\  $\ln|Z|$ is infinitely divisible; equivalently, the standard normal distribution considered as the distribution on the multiplicative group over $\R\setminus\{0\}$ is infinitely divisible. 
\end{abstract}

\maketitle




It is easy to see and very well known that the normal distribution is infinitely divisible. 
Thorin \cite{thorin77} showed that the log-normal distribution is infinitely divisible. This 
inspired an outpouring of related activities, results of which were presented in \cite{bondesson-book}. 

The log-normal distribution is the distribution of the random variable (r.v.) $e^X$, where $X$ is a normal r.v.  
In this note, we shall go from the normality in the opposite direction, in a sense. Let $Z$ be a standard normal  r.v. Then the distribution of the r.v.\ $U:=\ln|Z|$ may be referred to as the exp-normal distribution.  
Indeed, then $Z\D\vp e^U$; here and in what follows, $\D$ denotes the equality in distribution and $\vp$ stands for a Rade\-macher r.v., so that $\P(\vp=1)=\P(\vp=-1)=1/2$; we are assuming that $\vp$ is independent of all other r.v.'s mentioned in this note. 
One may observe that the probability density function, say $p$, of the exp-normal distribution is given by the formula 
\begin{equation*}
	p(u)=\sqrt{\frac2\pi}\,\exp\{u-e^{2u}/2\}
\end{equation*}
for all real $u$; so, the left tail of exp-normal distribution is asymptotically exponential, and the right tail is much lighter than any exponential or even normal tail. 

The main result of this note is as follows. 

\begin{theorem}\label{th:}
The exp-normal distribution is infinitely divisible; more specifically, 
\begin{align}
	\ln|Z|&\D\frac{\ln2}2-E_0-\sum_{j=1}^\infty\Big[\frac{E_j}{2j+1}-\frac12\ln\Big(1+\frac1j\Big)\Big]  \label{eq:} \\ 
&=-\frac{\ga+\ln2}2-\tE_0-\sum_{j=1}^\infty\frac{\tE_j}{2j+1}, \notag
\end{align}
where 
$E_0,E_1,\dots$ are independent identically distributed (iid) standard exponential r.v.'s, $\tE_j:=E_j-\E E_j=E_j-1$ are their centered versions, and $\ga$ is Euler's constant. 

Equivalently, the standard normal distribution, considered as the distribution on the multiplicative group over $\R\setminus\{0\}$, is (multiplicatively) infinitely divisible; that is, for each natural $k$ there exist iid r.v.'s $W_1,\dots,W_k$ such that 
\begin{equation}\label{eq:multiplic}
	Z\D W_1\cdots W_k. 
\end{equation}
More specifically, 
\begin{equation*}
	W_1\D \vp\,
	\exp\Big\{\frac{\ln2}{2k}-G_{1/k,0} 
	-\sum_{j=1}^\infty\Big[\frac{G_{1/k,j}}{2j+1}-\frac1{2k}\ln\Big(1+\frac1j\Big)\Big]\Big\}, 
\end{equation*}
where $G_{1/k,0},G_{1/k,1},\dots$ are iid r.v.'s each with the gamma distribution \break  $\operatorname{Gamma}(1/k,1)$ with shape parameter $1/k$ and scale parameter $1$, and $\vp$ is a Rade\-macher r.v.\ 
(independent of $G_{1/k,0},G_{1/k,1},\dots$). 
\end{theorem}

\begin{proof}
For the characteristic function (c.f.) $f$ of the r.v.\ $\ln|Z|$ and all real $t$, using the substitution $u=z^2/2$ and Euler's product formula   
\begin{equation*}
	\Ga(z)=\frac1z\,\prod_{j=1}^\infty\frac{(1+\frac1j)^z}{1+\frac zj}
\end{equation*}
for $z\in\C\setminus\{0,-1,-2,\dots\}$ (cf.\ e.g.\ line~2 in the first multi-line display on page~4 in \cite{andrews}), we have 
\begin{align*}
	f(t)=\E e^{it\ln|Z|}&=\int_{-\infty}^\infty e^{it\ln|z|}\frac1{\sqrt{2\pi}}\,e^{-z^2/2}\dd z \\ 
	&=\frac2{\sqrt{2\pi}}\,\int_0^\infty z^{it}e^{-z^2/2}\dd z \\ 
	&=\frac{2^{it/2}}{\sqrt\pi}\,\int_0^\infty u^{it
	/2-1/2}e^{-u}\dd u \\ 
	&=\frac{2^{it/2}}{\sqrt\pi}\,\Ga\Big(\frac{1+it}2\Big)=2^{it/2}\Ga\Big(\frac{1+it}2\Big)
	\Big/\Ga\Big(\frac12\Big) \\ 
	&=\exp\Big\{it\frac{\ln2}2\Big\}\frac1{1+it}\, 
	\prod_{j=1}^\infty\frac{\exp\{\frac{it}2\,\ln(1+\frac1j)\}}{1+\frac{it}{2j+1}}. 
\end{align*}
Now formula \eqref{eq:} follows, because the c.f.\ of an exponential r.v.\ with mean $a>0$ is $t\mapsto\frac1{1-ita}$. Since the standard exponential distribution is infinitely divisible \big(being, for each natural $k$, the $k$-fold convolution of the gamma distribution $\operatorname{Gamma}(1/k,1)$\big), we have proved the first part of of Theorem~\ref{th:}; the second part of the theorem now easily follows. 
\end{proof}

The question of whether the factorization \eqref{eq:multiplic} is true for $k=2$ was asked at site \cite{sqrt-normal}.





\bibliographystyle{abbrv}

\bibliography{P:/pCloudSync/mtu_pCloud_02-02-17/bib_files/citations12.13.12}



\end{document}